\documentclass[12pt,reqno,oneside]{article}
\usepackage{enumerate}
\usepackage{amsmath,amsthm}
\usepackage{graphicx}
\usepackage{amssymb}
\usepackage[all,cmtip]{xy}
\usepackage{xypic}
\usepackage{mathrsfs}
\usepackage{verbatim}
\newtheorem{thm}{Theorem}[section]
\newtheorem{cor}[thm]{Corollary}
\newtheorem{lem}[thm]{Lemma}
\newtheorem{prop}[thm]{Proposition}

\newtheorem{claim}[thm]{Claim}

\theoremstyle{definition}

\newcommand*\diff{\mathop{}\!\mathrm{d}}

\newcommand{\Addresses}{{
  \bigskip
  \footnotesize

\textsc{Korteweg-de Vries Instituut, Universiteit van Amsterdam, Postbus 94248, 1090 GE
Amsterdam, The Netherlands}\par\nopagebreak
  \textit{E-mail address}: \texttt{Z.Zhou@uva.nl}

}}
\usepackage{tikz-cd}
\DeclareMathOperator{\ord}{ord}
\DeclareMathOperator{\Imm}{Im}
\begin{document}

\title{A bound on the genus of a curve with Cartier operator of small rank}
\author{Zijian Zhou}
\date{\vspace{-5ex}}
\maketitle 
\begin{abstract}
Ekedahl showed that the genus of a curve in characteristic $p>0$ with zero Cartier operator is bounded by $p(p-1)/2$. We show the bound $p+p(p-1)/2$ in case the rank of the Cartier operator is 1, improving a result of Re.
\end{abstract}
\section{Introduction}
In \cite{Ekedahl1987} Ekedahl gave a bound for the genus $g$ of an irreducible smooth complete curve over an algebraically closed field of characteristic $p>0$ with~zero~Cartier operator:
$g\leq p(p-1)/2$.
This bound is sharp and was generalized by Re to curves with Cartier operator of given rank \cite{rre}. He showed for hyperelliptic curves whose Cartier operator has rank $m$ the bound
$
g<mp+(p+1)/2
$, 
and for non-hyperelliptic curves
\begin{equation}\label{Re formular}
g\leq mp+(m+1)p(p-1)/2\, .
\end{equation}
He also showed that if the Cartier operator $\mathcal{C}$ is nilpotent and $\mathcal{C}^r=0$, then 
$$
g\leq p^r(p^r-1)/2\, .
$$
In this paper we give a strengthening of the result (\ref{Re formular}) of Re.
\begin{thm}\label{main theorem rk 1}
Let $X$ be an irreducible smooth complete curve of genus g over an algebraically closed field of characteristic $p>0$. If the rank of the Cartier operator of $X$ equals 1, we have
$
g\leq p+p(p-1)/2
$.
\end{thm}
This is sharp for example for $p=2$. In the case of higher rank we have the following result.
\begin{thm}\label{main theorem rk 2}
Let $X$ be an irreducible smooth complete curve of genus g over an algebraically closed field of characteristic $p>0$. If the rank of the Cartier operator of $X$ equals 2, and if $X$ possesses a point $R$ such that $|pR|$ is base point free, then
$
g\leq 2p+p(p-1)/2
$,
while if $X$ does not have such a point, one has the bound
$
g\leq 2p+(4p^2-5p)/2
$.

\end{thm}
\section{The Cartier operator and linear systems}
From now on, by a curve we mean an irreducible smooth complete curve over an algebraically closed field $k$ of characteristic $p>0$. For a curve $X$ with function field $k(X)$, Cartier \cite{MR0084497} defined  an operator on 
rational differential forms with the following properties:

1) $\mathcal{C}(\omega _1+\omega _2)=\mathcal{C}(\omega _1)+\mathcal{C}(\omega _2)\, ,$

2) $\mathcal{C}(f ^p\omega)=f\, \mathcal{C}(\omega)\, ,$

3) $\mathcal{C}(\diff f)=0\, ,$

4) $\mathcal{C}(\diff f/f)=\diff f/f  \, , $

\noindent where $f \in k(X)$ is non-zero. Moreover, recall that if $x$ is a separating variable of $k(X)$, any $f\in k(X)$ can be written as 
\begin{equation}\label{cartier define}
f=f_0^p+\ldots+f^p_{p-1}x^{p-1},~\text{with}~f_i\in k(X) \, .
\end{equation}
For a rational differential form $\omega = f\diff x$ with $f$ as in (\ref{cartier define}), we have $\mathcal{C}(\omega)=f_{p-1}\diff x$. In particular $\mathcal{C}^n(f^i\diff f)=f^{(i+1)/p^n-1}\diff f$ if $p^n|i+1$, and $\mathcal{C}^n(f^i\diff f)=0$ otherwise. Furthermore, for distinct points $Q_1$, $Q_2$ on $X$, if there is a rational differential form $\omega$ that ${\ord _{Q_1}(\omega)\geq p}$ and $\ord_{Q_2} (\omega)=p-1$, then by property $2)$ above we have $\ord_{Q_1}(\mathcal{C}(\omega))\geq 1$ and $\ord_{Q_2}(\mathcal{C}(\omega))=0$.

This operator $\mathcal{C}$ induces a map $\mathcal{C}: \, H^0(X,\Omega_X^1) \to H^0(X,\Omega_X^1)$ which is $\sigma ^{-1}$-linear, that is, it satisfies properties 1) and 2) above, with $\sigma$ denoting the Frobenius automorphism of $k$. We are interested in the relation between the rank of the Cartier operator, defined as $\dim_k\mathcal{C}(H^0(X,\Omega_X^1))$, and the genus~$g$.

Re showed that there is a relation between the rank of Cartier operator and the geometry of linear systems on a curve. We will list some results that we will use and refer for the proof to Re's paper \cite{rre}.
In the following, $X$ denotes a non-hyperelliptic curve and for $D$ a divisor on $X$,  we will denote by $H^i(D)$ the vector space $H^i(X,\mathcal{O} _X (D))$. 

We will say that a statement holds for a general effective divisor of degree$~n$ on $X$ if the statement is true for divisors in a nonempty open set of effective divisors of degree $n$ on $X$. We start with a few results of Re.
\begin{prop}\cite[Prop.\ 2.2]{rre} \label{ri.re prop2.2}
Let $X$ be a non-hyperelliptic curve with ${\rm rank}(\mathcal{C})=m$. Then for a general effective divisor $D=Q_1+\ldots+Q_{m+1}$ on $X$ with $\deg D=m+1$, one has
$$
h^0(pD)=1+h^0(pD-Q_{m+1}) \, .
$$
\end{prop}
This implies for a general divisor $D$ with $\deg D>{\rm rank}(\mathcal{C})$, that the linear system $|pD|$ is base point free. As a corollary, we have the following.
\begin{cor}\cite[Prop.\ 2.3]{rre}\label{gap}
If X is a non-hyperelliptic curve with zero Cartier operator, then $h^0(p\, Q)\geq 2$ for any point $Q$ on $X$. 
\end{cor}
The following lemma gives a way of estimating dimensions of linear systems.
\begin{lem}\cite[Lemma 3.1]{rre}\label{R.R}
Assume that $Q_1$ and $Q_2$ are general points on a non-hyperelliptic curve $X$ and that $D$ is a divisor. Then we have 
$$
h^0(pD+p\,Q_1+p\,Q_2)-h^0(pD+p\,Q_1)\geq h^0(pD+p\,Q_1)-h^0(pD) \, .
$$
\end{lem}

We now give a generalization of a result of Re.
\begin{prop}\label{divisor}
Let $D$,$E$,$F$ be effective divisors on non-hyperelliptic curve X such that

\noindent$(1)$ $|F|$ is base point free;\\
$(2)$ $D > 0$ and ${\rm Supp}(D) \cap {\rm Supp}(E)=\emptyset$;\\
$(3)$ There are points $Q_1,\ldots,Q_{m+1} \in {\rm Supp}(D)$ and a divisor $F_1 \in |F|$ such~that $\ord _{Q_i}(F_1)=1$ for $1\leq i \leq m+1$ and ${\rm Supp}(D)\cap {\rm Supp}(F_1)=~\{Q_1,\ldots,Q_{m+1}\}$;\\
$(4)$ For these points $Q_i$ one has $h^0(E+\sum\limits_{i=1}^{m+1}Q_i)=h^0(E)$;\\
$(5)$ $Q_i$ is not a base point of $|D+E+F|$ for $i =1,\ldots,m+1$ and there exists $s_1,\ldots,s_{m+1}\in H^0(D+E+F)$ such that
$$
\ord _{Q_i}(s_i)=0\, , ~
\ord _{Q_i}(s_j)\geq p, ~~i\neq j,~i,j=1,\ldots,m+1\, .
$$
Then we have
$$
h^0(D+E+F)-h^0(E+F)\geq h^0(D+E)-h^0(E)+m+1\, .
$$
\end{prop}
\begin{proof}
Let  $s_{F_1}\in H^0(F)$ with divisor $F_1$ and $s_D\in H^0(D)$ with divisor $D$.
We have a commutative diagram with exact rows:
\begin{equation}
\begin{tikzcd}
0 \arrow[r] & H^0(E) \arrow[d,"\cdot s_{F_1}"] \arrow[r,"\cdot s_D"] & H^0(D+E)  \arrow[d,"\cdot s_{F_1}"] \arrow[r] &H^0(\mathcal{O}_D) \arrow[d,"\cdot s_{F_1}|_D"]\\
0  \arrow[r] & H^0(E+F) \arrow[r,"\cdot s_D"] & H^0(D+E+F) \arrow[r] &H^0(\mathcal{O}_D)\, .
\end{tikzcd} \nonumber
\end{equation}
\begin{claim}\label{claim}
Multiplication by $s_{F_1}$ induces an injective map
$$
\begin{tikzcd}
H^0(D+E)/s_D \cdot H^0(E) \arrow[r,"s_{F_1}"] & H^0(D+E+F)/s_D \cdot H^0(E+F) \, .
\end{tikzcd} \nonumber
$$
\end{claim}
\noindent This Claim follows if
$$
s_{F_1}\cdot H^0(D+E) \cap s_D\cdot H^0(E+F)=s_{F_1}\cdot  s_D\cdot H^0(E)\, .
$$
Because of assumptions $(2)$ and $(3)$, the left hand side of this  equation is equal to 
$
s_D\cdot  s_{F_1}'\cdot H^0(E+\sum_{i=1}^{m+1}Q_i)
$,
where $s_{F_1}'=s_{F_1}/s_0$ for a section $s_0 \in H^0(\sum_{i=1}^{m+1}Q_i)$ with ${\rm div}(s_0)=\sum_{i=1}^{m+1}Q_i$. Then $(4)$ implies 
$
H^0(E+\sum_{i=1}^{m+1}Q_i)=s_0 \cdot H^0(E)
$.
The Claim follows.

By $(5)$, we have $s_1,\ldots, s_{m+1}$ such that for all $i,j$ with $i\neq j$ we have
$
\ord_{Q_i}(s_i)=0$ and $\ord_{Q_i}(s_j)\geq p$. Now we will show that $s_1,\ldots,s_{m+1}$ generate an $m+1$-dimensional subspace of $H^0(D+E+F)/s_D \cdot H^0(E+F)$ with zero intersection with $\Imm(s_{F_1})$. First we will prove the zero intersection part. Assume there exist $c_1,\ldots, c_{m+1}\in k$ such that $\xi =\sum_{i=1}^{m+1}c_is_i $ lies in $\Imm (s_{F_1})$. That means $\xi=s_{F_1}\cdot r+s_D\cdot t$ with some $r\in H^0(D+E)$ and $t\in H^0(E+F)$. If $c_1\neq 0$ then we obtain
$
\ord_{Q_1}(\xi)=0
$.
 However, because $\ord_{Q_1}(F_1)=\ord_{Q_1}(s_{F_1})=1$ and $\ord_{Q_1}(s_D)=\ord_{Q_1}(D)\geq 1$, we have
$
0=\ord_{Q_1}(\xi)=\ord_{Q_1}(s_{F_1}\cdot r+s_D\cdot t)\geq 1 
$, a contradiction if $c_1\neq 0$. Similarly, we can show $c_2=\cdots=c_{m+1}=0$. Then for any non-zero element $\xi$ in $<s_1,\ldots,s_{m+1}>$ one has $\xi\notin \Imm(s_{F_1})$.

Now for the linear independence of $s_1,\ldots,s_{m+1}$, if $\xi=\sum_{i=1}^{m+1}{c_is_i}$ lies in $s_D\cdot H^0(E+F)$, then $\xi=s_D\cdot t$ with $t \in H^0(E+F)$ and we apply the same argument on the orders at $Q_i$ as above with $r=0$. Then we find $c_i=0$ for $i=1,\dots,m+1$. So $s_1,\ldots,s_{m+1}$ are linearly independent in $H^0(D+E+F)/s_D\cdot H^0(E+F)$. By the injectivity Claim \ref{claim} above we then have
$$
h^0(D+E+F)-h^0(D+E)\geq h^0(E+F)-h^0(E)+m+1\, .
$$
\end{proof}  

\section{Proofs of the Theorems (\ref{main theorem rk 1}) and (\ref{main theorem rk 2})}
Before giving the proofs of theorems, we need several lemmas on the relation between the rank of the Cartier operator and geometrical properties of linear systems on a curve.
\begin{lem}\label{existence of base point free divisor of mine}
Let $X$ be a non-hyperelliptic curve with$~{{\rm rank}(\mathcal{C})=m\geq 1}$. Then there exists points $Q_1, \dots, Q_m$ on $X$ such that with $D=\sum_{i=1}^mQ_i$ we have
\begin{align*}
h^0(pD)=1+h^0(pD-Q_m)\, .
\end{align*}
\end{lem}
\begin{proof}
Suppose that $\omega_1,\dots,\omega_m$ are differentials that generate $\Imm (\mathcal{C})$. Assume the lemma is not true, that is, for any $m$-tuple $\alpha=(Q_1,\dots,Q_m)$, we have with $D=\sum_{i=1}^{m}Q_i$ that $h^0(p\, D)=h^0(p\,D-Q_m)$.
Then by Serre duality and Riemann-Roch, there exists a $\omega_{D} \in H^0(X,\Omega_X^1)$ that
\begin{align} \label{omega_D 1}
\ord _{Q_i}(\omega_{D}) \geq p, ~1\leq i\leq m-1,~
\ord _{Q_m}(\omega_{D}) =p-1\, .
\end{align}
Let  $\eta := \mathcal{C}(\omega_{D})=\sum_{i=1}^m{\lambda_i \, \omega_i}$ with $\lambda_i\in k$. Then one has 
\begin{align}\label{eta_1}
\ord _{Q_i}(\eta) &\geq 1, ~1\leq i \leq m-1\, , ~
\ord _{Q_m}(\eta) =0\, .
\end{align}
Suppose now that $\omega_1,\ldots,\omega_m$ have a common base point $R$. Then define $Q_m=R$ and choose general points $Q_1,\dots,Q_{m-1}$ such that $Q_1,\dots,Q_{m-1},R$ form $m$ distinct points. Then with $D=\sum_{i=1}^{m-1}Q_i+R$ we have $h^0(pD)=h^0(pD-R)$, hence there exists a $\omega_D$ satisfying (\ref{omega_D 1}). Then $\eta=\mathcal{C}(\omega_D)$ satisfies (\ref{eta_1}) and we have $0=\ord_{Q_m}(\eta)=\ord_{Q_m}(\sum_{i=1}^m \lambda_i\omega_i)\geq 1$, a contradiction.

So we may assume that $\omega_1,\dots,\omega_m$ have no common base point. Choose a point $Q_1$ such that $\omega_1$ does not vanish at $Q_1$, but $\omega_2,\dots,\omega_m$ vanish at $Q_1$.
More generally, assume furthermore that we have $Q_1,\dots,Q_n$ such that $\ord_{Q_i}(\omega_i)=0$ and $\ord_{Q_i}(\omega_j)>0$ for $i=1,\dots,n$ and $i<j\leq m$.

If $\omega_{n+1},\dots,\omega_m$ have a base point $R$ different from $Q_i$ for $i=1,\dots,n$, then we choose $Q_{n+1},\dots,Q_{m-1}$ general distinct points, $Q_m=R$ and let $\alpha=(Q_1,\dots,Q_m)$. By assumption $h^0(p\, D)=h^0(p\, D-Q_m)$ for $D=\sum_{i=1}^mQ_i$, and we find a differential form $\omega_D$ satisfying (\ref{omega_D 1}) and therefore a form $\eta=\mathcal{C}(\omega_D)$ satisfying (\ref{eta_1}), again a contradiction.

So we may assume that $\omega_{n+1},\dots,\omega_m$ do not have common base points except $Q_1,\dots,Q_n$. Choose now a point $Q_{n+1}$ different from $Q_1,\dots,Q_n$ such that $\omega_{n+1}$ does not vanish at $Q_{n+1}$, but $\omega_{n+2},\dots,\omega_m$ all vanish at $Q_{n+1}$. By induction on $n$, we find points $Q_1,\dots,Q_{m-1}$ with
$
\ord_{Q_i}(\omega_i)=0$ and $\ord_{Q_i}(\omega_j)\geq 1$ for $j>i$ and $j=2,\dots,m$.

Now if $\omega_m$ has a zero distinct from $Q_i$ for $i=1,\dots,m-1$, say $Q_m$, we let
$
\alpha=(Q_1,\dots,Q_m)
$ and $D=\sum_{i=1}^mQ_i$.
The assumption $h^0(pD)=h^0(pD-Q_m)$ gives us a differential form $\omega_D$ and $\eta =\mathcal{C}(\omega_{D})=\sum_{i=1}^m \lambda_i \, \omega_i$. By (\ref{eta_1}) we have 
$
0=
\ord_{Q_m}(\eta)=\ord_Q(\lambda_m \omega_m)\geq 1
$, a contradiction. So $\omega_m$ has no zeros outside $Q_1,\dots,Q_{m-1}$.

Now $\deg(\omega_m)=2g-2\geq m$ for $g\geq 2$, so $\omega_m$ vanishes at one $Q_i$ with multiplicity larger than one, say $Q_{m-1}$. Then with $D=\sum_{i=1}^{m-2}Q_i+2Q_{m-1}$ we have $h^0(pD)=h^0(pD-Q_{m-1})$, giving us a differential form $\omega_D$, and $\eta=\mathcal{C}(\omega_D)=\sum_{i=1}^m\lambda_i\, \omega_i$. Then we have
$
\ord_{Q_i}(\eta)\geq 1$ for $i=1,\dots,m-2$ and
$\ord_{Q_{m-1}}(\eta)=1
$.
However, by the induction assumption
\begin{align*}
\ord_{Q_i}(\omega_i)&=0\, ,
\ord_{Q_i}(\omega_j)\geq 1, ~1\leq i<j\leq m-1\, ,\\
\ord_{Q_l}(\omega_m)&\geq 1\, ,
\ord_{Q_{m-1}}(\omega_m)\geq 2,~ l=1,2,\dots,m-2\, .
\end{align*}
So we must have $\lambda_i=0$ for $i=1,\dots,m-1$ and 
$
\ord_{Q_{m-1}}(\eta)\geq 2
$, and we therefore find
$
h^0(pD)=1+h^0(pD-Q_m)
$.
\end{proof}
\begin{cor}\label{rank1 point}
Let $X$ be a non-hyperelliptic curve. If the Cartier operator has ${\rm rank}(\mathcal{C})=1$, there exists a point $R$ of $X$ such that 
$
h^0(pR)=1+h^0((p-1)R) 
$.
\end{cor}
Combining Lemma \ref{existence of base point free divisor of mine} above and Proposition \ref{divisor}, we have the following result. We denote the canonical divisor (class) by $K_X$.
\begin{cor}\label{Cor rank 1}
Let $X$ be a non-hyperelliptic curve with ${\rm rank}(\mathcal{C})=1$ and let $T_n$ be a general effective divisor of degree$~n$. Put $E=p\, T_n$ and suppose $R$ is a point of $X$ with $h^0(pR)=1+h^0((p-1)R)$. Then the following holds.

\noindent $~i)$ If $h^0(K_X-E)\leq 1$, one has for general points $Q_1$, $Q_2$
\begin{align*}
h^0(E+pR+\sum_{i=1}^2p\,Q_i)-h^0(E+\sum_{i=1}^2p\,Q_i)=p\, .
\end{align*}
$ii)$ If $h^0(K_X-E)\geq 2$, one has for general points $Q_1$, $Q_2$ 
\begin{align*}
h^0(E+pR+\sum_{i=1}^2p\,Q_i)-h^0(E+\sum_{i=1}^2p\,Q_i)\geq 2+h^0(E+pR)-h^0(E)\, .
\end{align*}

\end{cor}
\begin{proof}

 $i)$ If $h^0(K_X-E)=0$, i.e.\ $E$ is non-special, Riemann-Roch implies statement $i)$.
If $h^0(K_X-E)=1$, we choose $Q_1$ a non-base point of $|K_X-E|$, then
$
h^0(K_X-E-Q_1)=0$, hence $h^0(K_X-E-p\,Q_1)=0\, .
$
Therefore
$
h^0(K_X-E-\sum _{i=1}^2 p\,Q_i)=h^0(K_X-E-\sum _{i=1}^2 p\,Q_i-p\,R)=0
$ and by Riemann-Roch we have 
$
h^0(E+\sum _{i=1}^2 p\,Q_i+p\,R)-h^0(E+\sum _{i=1}^2 p\,Q_i)=p
$.

\noindent$ii)$ If $h^0(K_X-E)\geq 2$, we write
$
D=p\, Q_1+p\, Q_2 ,~
E=p\, T_n$ and $
F=p\,R
$
and we proceed to verify the conditions $(1)-(5)$ of Proposition \ref{divisor} in this case.
Conditions $(1)$ and $(2)$ are easy consequences of the generality assumptions of $Q_1$, $Q_2$ and $R$.
For condition (3), if the linear system $|pR|$ induces a separable map to projective space, then we can choose $Q_1$ and $Q_2$ to be points where the map is smooth and find an effective divisor $F_1$ such that $\ord_{Q_1}(F_1)=\ord_{Q_2}(F_1)=1$.
If, on the contrary, the map induced by $|pR|$ is inseparable, then $\dim|R| \geq 1$, which is not true for curves of genus larger than zero.

Condition $(4)$ is satisfied once we choose $Q_1$ to be a non-base point of $|K_X-E|$ and $Q_2$ a non-base point of $|K_X-E-Q_1|$, since $h^0(K_X-E)\geq 2$. Then we have
$
h^0(K_X-E-\sum_{i=1}^2Q_i)=h^0(K_X-E)-2
$.

Condition $(5)$ holds as $|E+pR+pQ_1+pQ_2|$ is base point free by Proposition \ref{ri.re prop2.2} 
if $Q_1$ and $Q_2$ are general.
Furthermore by Proposition \ref{ri.re prop2.2}, we have 
$
h^0(E+pQ_i)=1+h^0(E+(p-1)Q_i)$ for $i=1,2 $. Then we obtain $s_1$ and  $s_2$ in $H^0(E+pR+pQ_1+pQ_2)=H^0(D+E+F)$ such that for all $i,j$ we have
$
\ord_{Q_i}(s_i)=0$ and
$\ord_{Q_i}(s_j)\geq p$ for $j\neq i
$.

Then we conclude by Proposition \ref{divisor} above.

\end{proof}
Now we can state some numerical consequences of Corollary \ref{Cor rank 1}.
\begin{cor}\label{cor2 rank 1}
Let $X$ be a non-hyperelliptic curve with ${\rm rank}(\mathcal{C})=1$. Denote by $D_n$ a general divisor of degree$~n$. Then for any integer $n\geq 1$, one has

$~~i)$ ~$p\geq h^0(pD_{2n})-h^0(pD_{2n-1})\geq \min(2n-1,p)$.

$~ii)$ ~For $1\leq n \leq \left  \lceil (p+1)/2\right \rceil$, one has
$
p\geq h^0(pD_{2n-1})-h^0(pD_{2n-2})\geq~{2n-2}
$.

$iii)$~ $pD_p$ is non-special, i.e. $h^0(K_X-pD_p)=0$.

$\,iv)$ \,For $1\leq n \leq \left  \lceil (p+1)/2\right \rceil$, one has
$
h^0(pD_{2n})-h^0(pD_{2n-2})\geq 4n-3
$.

$~\,v)$ \,For $1\leq n \leq \left \lceil (p+1)/2\right \rceil$, one has
$$
h^0(K_X-pD_{2n-2})-h^0(K_X-pD_{2n})\leq 2p-4n+3\, .
$$

$\,vi)$ ~$h^0(K_X-pD_{p-1})\leq 1$.

\end{cor}
\begin{proof}
$i)$ For $n\in \mathbb{Z}_{>0}$, one can always has$~p\geq h^0(pD_{2n})-h^0(pD_{2n-1})$. We will prove the second inequality in $i)$ by induction on $n$.

In the case $n=1$, by Proposition \ref{ri.re prop2.2}, for general points $Q_1$, $Q_2$ one has
\begin{align*}
h^0(pQ_1+pQ_2)=1+h^0(pQ_1+(p-1)Q_2)\, ,
\end{align*}
 and thus with $D_2=Q_1+Q_2$ and $D_1=Q_1$, we see $h^0(pD_2)\geq 1+h^0(pD_1)$.
Now we do induction and assume $h^0(pD_{2n-2})-h^0(pD_{2n-3})\geq 2n-3$. We apply Corollary \ref{Cor rank 1} with $E=pD_{2n-3}$ for $n\geq 2$. If $h^0(K_X-E)\leq 2$, then we have
$
h^0(pD_{2n})-h^0(pD_{2n-1})=p 
$.
Otherwise, Corollary \ref{Cor rank 1} implies
$$
h^0(pD_{2n})-h^0(pD_{2n-1})\geq 2+h^0(pD_{2n-2})-h^0(pD_{2n-3})\geq 2n-1 \, .
$$
and we are done.

\noindent $ii)$ The case $n=1$ is trivial. Assuming the assertion for $n-1$, we will prove
\begin{align}\label{equ cor3.4}
h^0(pD_{2n-1})-h^0(pD_{2n-2})\geq 1+h^0(pD_{2n-2})-h^0(pD_{2n-3})\, ,
\end{align}
and by $i)$ the right hand side is at least $2n-2$, which suffices for $iii)$. To prove the inequality (\ref{equ cor3.4}), take $D=p\, Q_1$, $E=p\, D_{2n-3}$ and $F=p\, R$ with the point $R$ satisfying $h^0(pR)=1+h^0((p-1)R)$ (see Corollary \ref{rank1 point}) and $Q_1$ a general point. We are going to verify the conditions $(1)-(5)$ of Proposition \ref{divisor} in the case of $m=0$. Conditions $(1)$ and $(2)$ are obvious by the property of $R$ and generality of $Q_1$. For condition $(3)$, the map induced by $|pR|$ is separable for curves of genus $g>0$. We can choose $Q_1$ to be a point where the map is smooth.

For condition $(4)$ we can choose $Q_1$ to be a non-base point of $|K_X-E|$ as it is non-empty. For condition $(5)$, as $E=pD_{2n-3}$ with $n\geq 2$, we have for any point $Q$ in ${\rm Supp}(E)$, $|pQ+pQ_1|$ is base point free due to Proposition \ref{ri.re prop2.2}. Then $|D+E+F|$ is base point free and by Proposition \ref{divisor} we have
\begin{align*}
h^0(pD_{2n-1})-h^0(pD_{2n-2})\geq 1+h^0(pD_{2n-2})-h^0(pD_{2n-3})\geq 2n-2\, .
\end{align*}
\noindent $iii)$ For $p$ odd, we let $n=(p+1)/2$ and apply $i)$ get 
$
h^0(pD_{p+1})-h^0(pD_p)\geq p
$. For $p=2$, we let $n=2$ and apply $ii)$ we also get $
h^0(pD_{p+1})-h^0(pD_p)\geq p
$. So we have
$
h^0(K_X-pD_{p})=h^0(K_X-pD_{p+1})
$. In other words, for a general point$~Q$, we see that $p\, Q$ lies in the base locus of $|K_X-pD_p|$. This can only happen when 
$
h^0(K_X-pD_{p})=0
$.
\noindent Property $iv)$ follows by combining $i)$ and $ii)$. Property $v)$ follows by $iv)$ and Riemann-Roch. For property $vi)$, by $ii)$ and $iii)$, it is known that
\begin{align*}
h^0(pD_p)-h^0(pD_{p-1})&\geq p-1\, ,~
h^0(K_X-pD_p)=0\, .
\end{align*} 
We have
\begin{align*}
h^0(K_X-pD_{p-1})&=h^0(pD_{p-1})+1-g-p(p-1) \\
&\leq h^0(pD_p)+1-g-p(p-1)-(p-1)\\
&=h^0(K_X-pD_p)+1=1\, .
\end{align*}
\end{proof}
Using the inequalities in Corollary \ref{cor2 rank 1}, we can easily prove Theorem \ref{main theorem rk 1}.

\begin{proof}[Proof of Theorem \ref{main theorem rk 1}]
We estimate $g=h^0(K_X)$ by 
\begin{align*}
h^0(K_X)&=\sum_{n=1}^{\left \lceil (p+1)/2\right \rceil}{(h^0(K_X-pD_{2n-2})-h^0(K_X-pD_{2n}))}+h^0(K_X-pD_{p-1}) \\
&\leq \sum_{n=1}^{\left \lceil (p+1)/2\right \rceil}(2p-4n+3)+1
=p+p(p-1)/2\, .
\end{align*}
\end{proof}

Our approach to the case ${\rm rank}(\mathcal{C})=2$ is similar, but there are differences due to the existence of special linear systems. We now give the analogue of Corollary \ref{Cor rank 1}.

\begin{cor}\label{Cor rank 2}
Let $X$ be a non-hyperelliptic curve with ${\rm rank}(\mathcal{C})=2$, and let $T_n$ be a general effective divisor of degree$~n$ and put $E=p\, T_n$. Let $Q_1$, $Q_2$, $Q_3$ be general points of $X$ and put $D=E+\sum_{i=1}^3p\, Q_i$.

\noindent $1)$ Assume there exists $R$ of $X$ such that $h^0(p\, R)=1+h^0((p-1)\,R)$.

$a)$ If $h^0(K_X-E)\leq 2$, then one has 
$
h^0(D+pR)-h^0(D)=p
$.

$b)$ If $h^0(K_X-E)\geq 3$, then one has
$$
h^0(D+pR)-h^0(D)\geq h^0(E+pR)-h^0(E)+3\, .
$$

\noindent $2)$ If there does not exists such a point $R$, we choose points $R_1$,$R_2$ satisfying $h^0(\sum_{i=1}^2p\, R_i)=1+h^0(\sum_{i=1}^2p\, R_i-R_2)$ and let $\deg E \geq 2p$.

$a)$ If $h^0(K_X-E)\leq 2$, then one has 
$
h^0(D+\sum_{j=1}^2pR_j)-h^0(D)=2p
$.

$b)$ If $h^0(K_X-E)\geq 3$, then one has
$$
h^0(D+\sum_{i=1}^2pR_j)-h^0(D)\geq h^0(E+\sum_{j=1}^2pR_j)-h^0(E)+3\, .
$$

\end{cor}

\noindent Note that in $2)$ we can choose such $R_1$ and $R_2$ by Lemma \ref{existence of base point free divisor of mine}. The proof is similar to the proof of Corollary \ref{Cor rank 1}. But we point out that~in the proof of part $2)$, instead of using the separable map induced by $|p\,R|$ in part $ii)$ of the proof of Corollary \ref{Cor rank 1}, we consider the map induced by $|p\,R_1+p\,R_2|$ with points $R_1$ and $R_2$. This map is separable, otherwise $\dim|R_1+R_2|\geq 1$,  contradicting that $X$ is non-hyperelliptic.

The following two corollaries are the analogues of Corollary \ref{cor2 rank 1}.
\begin{cor}\label{cor2 rank 2 case1}
Let $X$ be a non-hyperelliptic curve with ${\rm rank}(\mathcal{C})=2$. Denote by $D_n$ a general divisor of degree$~n$. If there exists a point $R$ of  $X$ that $|p\, R|$ is base point free, then for any integer $n\geq 1$, one has

\noindent ~~$i)$ $~p\geq h^0(pD_{3n})-h^0(pD_{3n-1})\geq \min(3\,n-2,p)$.\\
$~~ii)$ ~For $1\leq n \leq \left \lceil (p+2)/3 \right \rceil$, one has
$$
2\,p\geq h^0(pD_{3n-1})-h^0(pD_{3n-3})\geq  \max (6\,n-7,0)\, .
$$
$iii)$ ~$pD_{p+1}$ is non-special, i.e. $h^0(K_X-p\,D_{p+1})=0$.\\
$\,iv)$ ~For $1\leq n \leq \left \lceil (p+2)/3 \right \rceil$, one has
$
h^0(pD_{3n})-h^0(pD_{3n-3})\geq 9\,n-9
$.\\
$~\,~v)$ ~For $1\leq n \leq \left \lceil (p+2)/3 \right \rceil$, one has
$$
h^0(K_X-pD_{3n-3})-h^0(K_X-pD_{3n})\leq 3\,p-9\,n+9\, .
$$
$\,vi)$ ~$\,h^0(K_X-pD_{p-1})\leq 3$.

\end{cor}

\begin{cor}\label{cor2 rank 2 case 2}
Let $X$ be a non-hyperelliptic curve with ${\rm rank}(\mathcal{C})=2$. Denote by $D_n$ a general divisor of degree$~n$. If $X$ does not possess a point $R$ such that $|p\, R|$ is base point free, then for any integer $n\geq 1$, one has

\noindent ~~$i)$ ~$2\,p\geq h^0(pD_{3n})-h^0(pD_{3n-2})\geq \min(3\,n-2,2p)$.\\
$~~ii)$ ~For $2\leq n \leq \left\lceil (2p+2)/3 \right \rceil$, one has
$$
2\,p\geq h^0(pD_{3n-2})-h^0(pD_{3n-3})\geq  1\, .
$$
$iii)$ ~$pD_{2p}$ is non-special, i.e. $h^0(K_X-pD_{2p})=0$.\\
\noindent$\,iv)$ ~For $2\leq n \leq \left \lceil (2p+2)/3 \right \rceil$, one has
$
h^0(pD_{3n})-h^0(pD_{3n-3})\geq 3\,n-1
$. 

\hspace{2.5mm}For $n=1$, one has 
$
h^0(pD_{3})-h^0(pD_{0})\geq 1
$.\\
$~\,~v)$ ~For $2\leq n \leq \left \lceil (2p+2)/3 \right \rceil$, one has
$$
h^0(K_X-pD_{3n-3})-h^0(K_X-pD_{3n})\leq 3\,p-3\,n+1\, .
$$

\hspace{2.8mm}For $n=1$, one has
$
h^0(K_X)-h^0(K_X-pD_{3})\leq 3\,p-1
$.

\noindent$\,vi)$~ $\,h^0(K_X-pD_{2p-1})\leq p-1$.

\end{cor}
The proofs of two corollaries above are similar to the proof of Corollary \ref{cor2 rank 1} and therefore we omit these. The corollaries above now readily imply the proof of theorem in the case of ${\rm rank} (\mathcal{C})=2$.
\begin{proof}[Proof of Theorem \ref{main theorem rk 2}]

$(1)$ If $|p\,R|$ is base point free, then by Corollary \ref{cor2 rank 2 case1} we have
\begin{align*}
h^0(K_X)&=\sum_{n=1}^{\left \lceil (p-1)/3 \right \rceil}(h^0(K_X-pD_{3n-3})-h^0(K_X-pD_{3n}))+h^0(K_X-pD_{p-1})\\
&\leq \sum_{n=2}^{\left \lceil (p-1)/3 \right \rceil}(3\,p-9\,n+9)+1+3=p(p-1)/2+2\,p\, .
\end{align*}
$(2)$ Otherwise, by Corollary~\ref{cor2 rank 2 case 2} we have
\begin{align*}
h^0(K_X)&=\sum_{n=1}^{\left \lceil (2p-1)/3 \right \rceil}(h^0(K_X-pD_{3n-3})-h^0(K_X-pD_{3n}))+h^0(K_X-pD_{2p-1})\\
&\leq\sum_{n=2}^{\left \lceil (2p-1)/3 \right \rceil}(3\,p-3\,n+1)+3\,p-1+p-1=2\,p+(4\,p^2-5\,p)/3\, .
\end{align*}

\end{proof}
\bibliographystyle{alpha}

\Addresses
\end{document}